\documentclass[a4paper,11pt]{article}
\usepackage{amsmath,amsfonts,amsthm, amssymb,xspace,color}
\usepackage{makeidx}
\newcommand{\emptyword}{\varepsilon}
\newcommand{\shortlex}{{\sc Shortlex}\xspace}
\newcommand{\fsa}{{\sf fsa}\xspace}
\newtheorem{theorem}{Theorem}[section]
\newtheorem{proposition}[theorem]{Proposition}
\newcommand{\Z}{{\mathbb Z}}
\newcommand{\N}{{\mathbb N}}
\newcommand{\R}{{\mathbb R}}

\newcommand{\HH}{{\mathbb H}}
\newcommand{\SL}{{\mathrm{SL}}}
\newcommand{\cD}{{\mathcal D}}
\newcommand{\cH}{{\mathcal H}}
\newcommand{\cR}{{\mathcal R}}
\newcommand{\cX}{{\mathcal X}}
\newcommand{\Cay}{{\sf Cay}\xspace}
\newcommand{\Nil}{{\sf Nil}\xspace}
\newcommand{\Sol}{{\sf Sol}\xspace}
\newcommand{\Area}{{\sf Area}\xspace}

\newcommand{\kbmag}{{\sc kbmag}\xspace}
\newcommand{\gap}{{\sf GAP}\xspace}
\newcommand{\magma}{{\sf Magma}\xspace}
\def \rewrites#1 {\xrightarrow{#1}\,}
\def \seqrewrites#1 {\xrightarrow{#1}\!\!{}^*\,\,}
\def \slexle{<_{\rm slex}}
\def \mij {{\sf m_{ij}}}
\def \mjk {{\sf m_{jk}}}
\def \mik {{\sf m_{ik}}}
\def \mil {{\sf m_{il}}}
\def \mkl {{\sf m_{kl}}}
\def \WA {{\sf WA}}

\title{The development of the theory of automatic groups}
\author{Sarah Rees,\\ University of Newcastle,\\ Sarah.Rees@newcastle.ac.uk}
\makeindex

\begin{document}
\maketitle
\begin{abstract}
We describe the development of the theory of automatic groups.
We begin with a historical introduction, define the concepts
of automatic, biautomatic and combable groups,
derive basic properties, then explain how hyperbolic groups and the groups 
of compact 3-manifolds based on six of Thurston's eight geometries can be 
	proved automatic.
We describe software developed in Warwick to compute automatic structures, as
well as the development of practical algorithms that use those structures.
We explain how actions of groups on spaces displaying various notions of 
negative curvature can be used to prove automaticity or biautomaticity,
and show how these results have been used to derive these properties for 
groups in some infinite families (braid groups, mapping class groups, families of Artin groups, and Coxeter groups). 
Throughout the text we flag up open problems as well as problems that
remained open for some time but have now been resolved.
\end{abstract}


\newpage
\tableofcontents
\newpage
\section{Introduction}
\label{sec:intro}

This chapter describes the development of the theory of automatic groups. It aims to explain the definition, and put that into mathematical
and historical context, to detail what is known, give brief accounts of some
of the big problems in the subject that have already been solved, and describe
those problems that remain open.

Thurston is credited with the definition of automatic groups, and is one of six
authors of one of the primary early references of the subject \cite{ECHLPT};
but some of the foundations were laid in particular in work of Gromov on
hyperbolic groups\index{hyperbolic group} \cite{Gromov}, Cannon on properties of the fundamental
groups of compact hyperbolic manifolds \cite{Cannon84}, Gilman on groups with
rational cross-sections \cite{Gilman}. The standard reference is certainly the book \cite{ECHLPT}, but that is supplemented by some powerful results 
in \cite{BGSS,GerstenShort90,GerstenShort91}, while 
Farb's article \cite{Farb92} gives a useful and readable overview of early development of the subject.

The definition of an automatic group\index{automatic!group} was originally designed to identify properties of a 
group that were observed in the fundamental groups of compact hyperbolic 
3-manifolds\index{3-manifold}, and which facilitated computation with those groups.
Such groups are finitely generated. When a group is automatic, its associated 
automatic structure allows the elements of the group to be represented
as strings belonging to a particularly well structured set of strings, 
for which certain computations can be easily performed using finite state automata, as we shall see below.

Within this introductory section, 
we shall give some historical background, then define the notation and terminology that we shall need in the remainder of this chapter. 
Section~\ref{sec:automatic} contains the definition of an automatic group, 
identifies the basic properties, and describes the most natural examples,
and non-examples. 
Section~\ref{sec:computing} describes computation with automatic groups, how automatic structures may be computed, how they may, and have been, used. 
Section~\ref{sec:actions} describes how automaticity or biautomaticity of a group
may be deduced from the geometry of a space on which the group has a good action. 
Section~\ref{sec:eg_families} describes the derivation of results proving automaticity or 
biautomaticity of groups in some well known families of group, which often used  techniques or results described in Section~\ref{sec:actions}.
Finally Section~\ref{sec:open_problems} describes some problems that remain open.

\subsection{Historical background}
\label{sec:historical}
Alongside Thurston, it is natural to indentify 
Cannon, Epstein and Holt as the key figures in the early development of automatic groups.  
Much of the information in this section comes from discussion with these three people \cite{Cannon21,Epstein20,Holt21}, or can be found in the preface of the standard 
reference \cite{ECHLPT}.

Cannon's article \cite{Cannon21} identifies the International Congress of Mathematicians in Helsinki in 1978 as a location at which key ideas that influenced the development of the concept of an automatic group were discussed.

In his plenary address, Thurston discussed the construction of geometric structures on a 3-manifold $M$, and the tesselation of its universal cover $\tilde{M}$
by a structure dual to the Cayley graph of $\pi_1(M)$.
Thurston's geometrisation conjecture \cite{Thurston82}, subsequently proved by Perelman, claimed that every closed 3-manifold was geometrisable, that is, admitted a
canonical decomposition
into pieces each admitting one of eight types of geometric structure.

In his article \cite{Cannon21}, Cannon attributes to Thurston at that conference
the conjecture that
the growth series of a group $G$ acting discretely, cocompactly and isometrically on a finite dimensional hyperbolic space $\HH_n$
should be a rational function. Cannon proved that conjecture in \cite{Cannon84},
where he identified features of $\HH_n$ within the Cayley graph $\Cay(G,X)$ for $G$ with respect to a finite generating set $X$.
In particular, he proved that $\Cay(G,X)$ admits finitely many types of ``cones'' on geodesics, and deduced from this
the rationality of the growth function of $G$.
Cannon also proved that the word\index{word problem} and conjugacy problems\index{conjugacy problem} for $G$ could be solved 
using analogues of Dehn's algorithms for those in hyperbolic surface groups \index{hyperbolic group}.
Gromov's 1987 article \cite{Gromov} defined a combinatorial notion of
hyperbolicity for a graph, and hence for a group (via its Cayley graph),
and generalised Cannon's results to groups satisfying this definition of hyperbolicity.
There is a substantial body of material studying (Gromov) hyperbolic groups\index{hyperbolic group},
in particular \cite{AlonsoEtAl}.

Thurston realised that the finiteness of the set of cone types in one of
Cannon's groups of hyperbolic isometries allowed the construction of a 
finite state automaton recognising the set of geodesic words within the group;
rationality of the growth function is an immediate consequence of
that set of words being the language of a finite state automaton.
``Fellow travelling'' properties of quasi-geodesic paths in $\HH_n$ that had been recognised by Cannon allowed the construction of further automata that recognised
right multiplication in the goup by a generator.

Now Thurston defined the concept of an automatic group\index{automatic!group}.
He called a group with finite generating set $X$ {\em automatic} if it
possessed a representative set of words $L$ over $X$, such that one finite
state automaton recognised the words in $L$, and other automata
recognised pairs of words in $L$ related in the group under right multiplication by the generators in $X$. 
Very early on, groups of this type were known as {\em regular groups} \cite{Holt21}. 
But this terminology conflicted with other uses of the term regular, and so was soon changed.

Initially,
in particular in \cite{ECHLPT,BGSS}, 
the study of the family of automatic groups 
was largely driven by the desire to find within it the groups of the geometrisable 3-manifolds,
and hence to harness computational techniques that were provided by the association of automatic groups with regular languages. 
Epstein realised very early on that any automatic group must be finitely presented,
while Thurston deduced that any such group had
quadratic Dehn function and hence word problem soluble in quadratic time.
Epstein and Holt in Warwick worked, together with the author of this chapter,
to develop practical procedures to (attempt to)
build automatic structures for finitely presented groups, and to compute
within the groups using those structures.

\subsection{Mathematical background and notation}
\label{sec:background_notation}
All the groups that we consider will be finitely generated. If $X$ is a finite
generating set for a group $G$, then we write
$G = \langle X \rangle$.
In that case every element of $G$ can be represented as a product (or string) of
elements of $X$ and their inverses. We denote by $X^{-1}$ the set of
symbols $x^{-1}$ for which $x \in X$, and then by
$X^{\pm}$ the disjoint union of $X$ and $X^{-1}$; every non-identity
element of $G$
can now be described as a string of elements of $X^{\pm}$.
The identity element, which we denote by $1$, can be described as a product of
length 0.

Given a finite set $A$, 
we define a {\em string} $w$ over $A$ to be a sequence
$a_1a_2\cdots a_n$ with $a_i \in A$, and call $n$ the {\em length} of $w$,
denoted by $|w|$; we may alternatively use the term {\em word} over $A$ rather than string.
A subsequence $a_{i}a_{i+1}\cdots a_j$ of $w$ is called a {\em substring} or {\em subword}.
We write $w(i)$ for the {\em prefix} $a_1\cdots a_i$
of $w$.
We call the string or word of length $0$ over $A$ the {\em empty string} or {\em empty word} and denote that by $\emptyword$.
As is standard, we 
denote by $A^+$ the set of all strings over $A$ of finite length $>0$ and by
$A^*$ the union $A^+ \cup \{ \emptyword \}$.
Given an ordering of the elements of $A$, we define the shortlex ordering on $A^*$ as follows: for words 
$u=x_1\ldots x_r$ and $v=y_1\ldots y_s$ , we define $v \slexle u$ if 
$|v|<|u|$, or if $|u|=|v|$ and for some $i$,
$y_1=x_1,\ldots,y_{i-1}=x_{i-1}$ but $y_i<x_i$.

When $X$ is a generating set for a group $G$, and $w \in (X^\pm)^*$,
it is often convenient to abuse notation and use $w$ to indicate not only 
that string over $X^\pm$ but also the group element that the string 
represents; if $w,v \in X^\pm$,
we write $w=v$ to denote that $w,v$ are identical as strings,
and $w=_G v$ to denote that $w,v$ represent the same group element.
If $g \in G$, we denote by $|g|$ the length of the shortest word over $X^\pm$
that represents $g$.
Suppose that $\Cay=\Cay(G,X)$ is the {\em Cayley graph}\index{Cayley graph} of $G$ over $X$, 
that is 
the graph with vertex set $G$ and, for each $g \in G$, $x \in X$,
directed edges labelled $x$ and $x^{-1}$
connecting the ordered pairs of vertices $(g,gx)$ and $(gx,g)$. 
Then for each $g \in G$, a path labelled by $w$ joins the vertex $g$ of $\Cay$
to the vertex $gw$; we shall represent that path as ${}_g w$.

When $G$ is finitely generated by $X$, we define a {\em language}\index{language} for 
$G$ over $X$ to be
a subset of $(X^\pm)^*$ that contains at least one representative of each
element of $G$, that is, that maps onto $G$ under the map assigning each 
product over $X$ to the element it represents.

For the free group $F_n$ on a set $X$ of $n$ generators $x_1,\ldots,x_n$, a language is provided by the set of all {\em freely reduced} words  of length $\geq 0$
over $X^\pm$, that is, the set of all words within which no
subword $x_ix_i^{-1}$ or $x_i^{-1}x_i$ appears.
For the free abelian group $\Z^n$ on the same set of $n$ generators,
a language is provided by the set of all words of the form
$x_{i_1}^{r_1}x_{i_2}^{r_2}\cdots x_{i_k}^{r_k}$, with $k \geq 0$,
$i_1<i_2<\cdots <i_k$ 
and $r_i \in \Z \setminus \{0\}$.
In each of these two examples the language provides a unique representative for each group element.

Each of the two languages just described is an example of a {\em regular language}\index{language!regular},  that is, it is the set $L(M)$ of strings accepted by a {\em finite state automaton} (\fsa) $M$ with alphabet $\{x_1^{\pm 1},x_2^{\pm 1}\,\ldots, x_n^{\pm 1}\}$. Finite state automata provide
standard models of bounded memory computation and are defined and studied in \cite{HopcroftUllman}. It is common to
represent a finite state automaton $M$ with alphabet $A$ as a finite directed graph,
with each directed edge labelled by one or more elements of $A$, 
one vertex identified as the {\em start}, and a subset of the vertices selected as 
{\em accepting}. A word $w$ is then accepted by $M$ if it labels at least one 
directed path from the start to an accepting vertex; if there is no such path, or if
the end point ({\em target}) of every such path is a non-accepting vertex then $w$ is not accepted. It is standard to call the vertices of $M$
its {\em states}, the directed edges its {\em transitions} and the set of accepted words its {\em language} $L(M)$.
In the cases where $n=2$,
the languages described above for the free and free abelian groups over $\{a,b\}$
are accepted by the two
finite state automata shown in Figure~\ref{fig:fsa_F2Z2};
in each diagram, following convention, the start state is indicated by an arrow, and the
accepting states are ringed. In each of the two examples, each of the five states 
shown in the diagram is accepting, but a further {\em failure state} is not 
shown, which constitutes a sixth state; the failure state is non-accepting, 
any transitions not shown in the diagram are assumed to be to that failure 
state, and all transitions from the failure state are to the failure state.

\setlength{\unitlength}{1.0pt}
\begin{figure}[h]
\begin{minipage}[l]{2.2in}
\begin{picture}(260,160)(-80,-80)
\put(-15,-10){\vector(3,2){10}}
\put(0,0){\circle*{4}}
\put(0,0){\circle{10}}
\put(50,0){\circle*{4}}
\put(50,0){\circle{10}}
\put(-50,0){\circle*{4}}
\put(-50,0){\circle{10}}
\put(0,50){\circle*{4}}
\put(0,50){\circle{10}}
\put(0,-50){\circle*{4}}
\put(0,-50){\circle{10}}
\put(0,0){\line(1,0){50}}
\put(25,0){\vector(1,0){0}}
\put(25,2){$a$}
\put(0,0){\line(-1,0){50}}
\put(-25,0){\vector(-1,0){0}}
\put(-25,2){$a^{-1}$}
\put(0,0){\line(0,1){50}}
\put(0,25){\vector(0,1){0}}
\put(2,25){$b$}
\put(0,0){\line(0,-1){50}}
\put(0,-25){\vector(0,-1){0}}
\put(2,-25){$b^{-1}$}
\put(65,0){\oval(30,30)}
\put(80,0){\vector(0,1){0}}
\put(72,0){$a$}
\put(-65,0){\oval(30,30)}
\put(-80,0){\vector(0,1){0}}
\put(-77,0){$a^{-1}$}
\put(0,65){\oval(30,30)}
\put(0,80){\vector(1,0){0}}
\put(-2,82){$b$}
\put(0,-65){\oval(30,30)}
\put(0,-80){\vector(1,0){0}}
\put(-2,-78){$b^{-1}$}
\put(0,0){\oval(100,100)[tr]}
\put(30,50){\vector(-1,0){0}}
\put(33,53){$b$}
\put(0,0){\oval(100,100)[br]}
\put(30,-50){\vector(-1,0){0}}
\put(35,-60){$b^{-1}$}
\put(0,0){\oval(100,100)[tl]}
\put(-30,50){\vector(1,0){0}}
\put(-35,53){$b$}
\put(0,0){\oval(100,100)[bl]}
\put(-30,-50){\vector(1,0){0}}
\put(-35,-60){$b^{-1}$}
\put(0,50){\line(1,-1){50}}
\put(25,25){\vector(1,-1){0}}
\put(25,25){$a$}
\put(0,50){\line(-1,-1){50}}
\put(-25,25){\vector(-1,-1){0}}
\put(-36,25){$a^{-1}$}
\put(0,-50){\line(1,1){50}}
\put(25,-25){\vector(1,1){0}}
\put(25,-30){$a$}
\put(0,-50){\line(-1,1){50}}
\put(-25,-25){\vector(-1,1){0}}
\put(-41,-30){$a^{-1}$}
\end{picture}
\end{minipage}
\begin{minipage}[r]{2.2in}
\begin{picture}(260,160)(-100,-80)
\put(-15,-10){\vector(3,2){10}}
\put(0,0){\circle*{4}}
\put(0,0){\circle{10}}
\put(50,0){\circle*{4}}
\put(50,0){\circle{10}}
\put(-50,0){\circle*{4}}
\put(-50,0){\circle{10}}
\put(0,50){\circle*{4}}
\put(0,50){\circle{10}}
\put(0,-50){\circle*{4}}
\put(0,-50){\circle{10}}
\put(0,0){\line(1,0){50}}
\put(25,0){\vector(1,0){0}}
\put(25,2){$a$}
\put(0,0){\line(-1,0){50}}
\put(-25,0){\vector(-1,0){0}}
\put(-25,2){$a^{-1}$}
\put(0,0){\line(0,1){50}}
\put(0,25){\vector(0,1){0}}
\put(2,25){$b$}
\put(0,0){\line(0,-1){50}}
\put(0,-25){\vector(0,-1){0}}
\put(2,-25){$b^{-1}$}
\put(65,0){\oval(30,30)}
\put(80,0){\vector(0,1){0}}
\put(72,0){$a$}
\put(-65,0){\oval(30,30)}
\put(-80,0){\vector(0,1){0}}
\put(-77,0){$a^{-1}$}
\put(0,65){\oval(30,30)}
\put(0,80){\vector(1,0){0}}
\put(-2,82){$b$}
\put(0,-65){\oval(30,30)}
\put(0,-80){\vector(1,0){0}}
\put(-2,-78){$b^{-1}$}
\put(0,0){\oval(100,100)[tr]}
\put(30,50){\vector(-1,0){0}}
\put(33,53){$b$}
\put(0,0){\oval(100,100)[br]}
\put(30,-50){\vector(-1,0){0}}
\put(35,-60){$b^{-1}$}
\put(0,0){\oval(100,100)[tl]}
\put(-30,50){\vector(1,0){0}}
\put(-35,53){$b$}
\put(0,0){\oval(100,100)[bl]}
\put(-30,-50){\vector(1,0){0}}
\put(-35,-60){$b^{-1}$}
\end{picture}
\end{minipage}
\caption{\fsa $M_1,M_2$ giving the languages described in the text for $F^2$ and $\Z^2$}
\label{fig:fsa_F2Z2}
\end{figure}
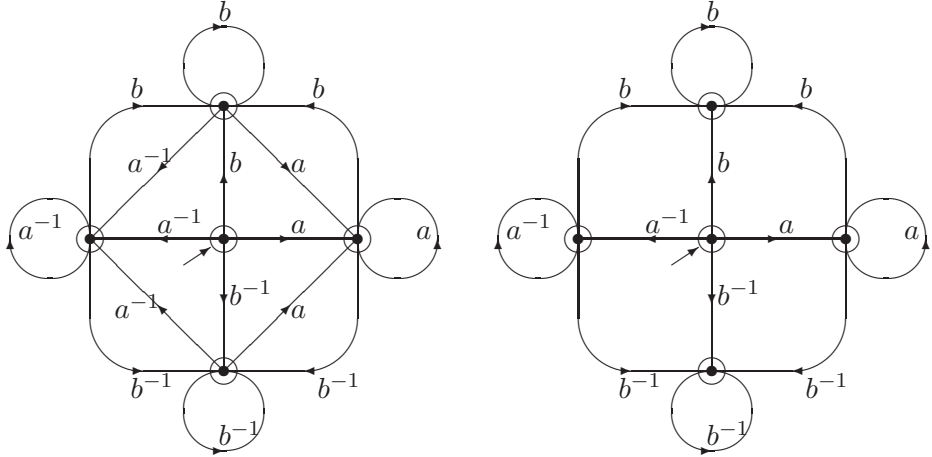

\section{Automatic groups}
\label{sec:automatic}
\subsection{Definition of an automatic group}
\label{sec: defn}
Now suppose that $G$ is a group with finite generating set $X$.
For $k \in \N$,
words $w,v$ over $X^\pm$ are said to {\em $k$-fellow travel}\index{fellow travel} in $G$
if for each $i \leq \max\{ |w|,|v|\}$ the distance between the vertices 
$w(i)$ and $v(i)$ of $\Cay=\Cay(G,X)$ (using the graph metric) is at most $k$.
Equivalently, we say that the paths ${}_1w$ and ${}_1v$ of $\Cay$
 $k$-fellow travel.
A
group $G$ with finite generating set $X$ is defined to be {\em automatic}\index{automatic!group}
over $X$
if
\begin{description}
\item[A1] there is a language $L$ for $G$ over $X$ that is regular,
\item[A2] there is an integer $k$ such that,
for each $y \in X \cup \{1\}$, 
		and for any $w,v \in L$ with  $wy=_G v$, the paths ${}_1w,{}_1v$ $k$-fellow travel in $\Cay$.
\end{description}
We call $L$ the {\em language}\index{language}, the \fsa accepting $L$ the {\em word acceptor}\index{word acceptor} and $k$ the {\em fellow traveller constant}\index{fellow travel!constant}
of an {\em automatic structure}\index{automatic!structure} for $G$.

The \fsa $M_1$ illustrated in Figure~\ref{fig:fsa_F2Z2} is the word acceptor
of an automatic structure with fellow traveller constant 1 for $F_2$ over $\{a,b\}$; 
each element of the group has a unique representative in the language, 
and given two words $w,v \in L(M_1)$ 
and $y \in \{a^{\pm 1},b^{\pm}\}$
with $wy=_{F_2} v$, 
one of the words is a maximal prefix of the other,
and so the words $1$-fellow travel in $G$.

Similarly, the \fsa $M_2$ of Figure~\ref{fig:fsa_F2Z2} is the word acceptor
of an automatic structure with fellow traveller constant 2 for $\Z^2$ over $\{a,b\}$.
Again each element of the group has a unique representative in the language,
and given two words $w,v \in L(M_2)$ 
and $y \in \{a^{\pm 1},b^{\pm 1}\}$
with $wy=_{\Z^2}v$, corresponding vertices on the paths ${}_1w$ and ${}_1v$
in $\Cay(\Z^2,\{a,b\})$ are joined in the graph by a path of length 1 or 2.
The language $L(M_2)$ is the set of
all \shortlex minimal geodesic representatives of group elements; we call this a 
{\em shortlex automatic structure}\index{automatic!shortlex} for $\Z^2$.
Note that we can define a similar shortlex automatic structure for $\Z^n$.

In the
definition of automaticity given in \cite{ECHLPT} 
the condition A2 given above is replaced by the following condition:

\begin{description}
\item[A2']
For each $y \in X \cup \{ \{1\}\}$,
the set of pairs $(w,v)$ for which $w,v \in L$ and $wy =_G  v$
is a regular language when viewed as a set of strings over the alphabet
of pairs $\{(a,b):a,b \in X^\pm \cup \{\$\} \}$;
the character $\$$ is a {\em padding symbol} used to deal with the situation where $|w|\neq |v|$, in which case
the shorter of the two words is padded with $\$$s at its end.
\end{description}
The automata recognising the regular languages just described 
are known as the {\em multiplier automata}\index{multiplier automaton} of the automatic structure,
usually denoted by $M_y$, for each choice of $y$.

In the presence of A1 the conditions A2 and A2' are equivalent.
This is a consequence of the fact that the $k$-fellow travelling of a pair
of words $w,v$ can be tracked by an automaton whose state set $\cD$ 
corresponds to a set of
words of length at most $k$; a pair of words $(w,v)$ is accepted by that automaton so long as all the products
$w'^{-1}v'$ associated with prefixes
$w':= w(i),v':=v(i)$ of $w,v$ are represented by words in $\cD$. 
We call such an automaton a {\em word difference machine}\index{word difference!machine}, and the associated
set $\cD$ its corresponding set of {\em word differences}.

Where $G$ is automatic over its finite generating set $X$, with automatic
structure $L,k$, then $G$ is said to be {\em biautomatic}\index{biautomatic!group} (and $(L,k)$ to be a 
{\em biautomatic structure}\index{biautomatic!structure} for $G$) if 
the additional condition A3 is satisfied:
\begin{description}
\item[A3] 
for each $y \in X$, 
and for any $w,v \in L$ with  $yw=_G v$, 
		the paths ${}_yw,{}_1v$ $k$-fellow travel in $\Cay(G,X)$.
\end{description}
This further fellow traveller condition can be expressed in terms of
\fsa that recognise left multiplication, usually denoted by ${}_yM$, for $y \in X$.
It is an open question whether all automatic groups are biautomatic.

The concept of automaticity can be generalised to one of {\em asynchronous automaticity}\index{automatic!asynchronously}
by replacing the fellow traveller condition by an {\em asynchronous fellow travel condition}\index{fellow travel!asynchronously}; for two words $w,v$ to asynchronously fellow travel within a group $G$ it is the distance between
vertices $w(j_i)$ and $v(k_i)$ that must be bounded, where, for some $m \geq \max(|w|,|v|)$, the sequences $(j_0,j_1,\ldots,j_m)$ and $(k_0,\ldots,k_m)$
are both increasing sequences of integers, with $j_0=k_0=0$, $j_m=|w|$, 
$k_m=|v|$, and for $0\leq l <m$, $j_{l+1}-j_l$ and $k_{l+1}-k_l$ are in $\{0,1\}$.
Asynchronous automaticity is certainly a more general concept than automaticity,
and it is satisfied by examples such as the Baumslag--Solitar groups which are certainly not automatic.

It is fairly standard to call a language $L$ for a group $G$ that satisfies
the condition A2 (but not necessarily A1) a {\em combing}\index{combing} for $G$, and a 
language that satisfies both A2 and A3 a {\em bicombing}\index{bicombing} for $G$; however some
authors use these terms differently, e.g. impose additional (geometric) conditions on $L$.
Again, the fellow travelling condition can be replaced by an asynchronous one, in order to define asynchronous combings and bicombings.
The basic properties of combable groups\index{combable group} are studied 
in \cite{Bridson03}, where it is proved that non-automatic combable groups 
exist (answering a question posed in \cite{ECHLPT}), as well as combable groups that are not bicombable.

Given an automatic (or biautomatic) structure $(L,k)$
for a group $G$, it is straightforward (using well known properties of 
regular languages, such as the ``Pumping lemma'' \cite{HopcroftUllman}) to modify 
the structure and
achieve a new automatic structure with particular properties.
For instance we can achieve a structure in which every element of $G$ has
a unique representative (a structure {\em with uniqueness}\index{automatic!with uniqueness}) a {\em prefix closed}\index{automatic!prefix closed} structure 
in which the language contains every prefix of every one of its
elements,
a {\em quasigeodesic}\index{automatic!quasigeodesically} structure in
which every element is represented by a $(\lambda,\epsilon)$-quasigeodesic,
We note that a word $w$ representing an element $g$ of a group $G$ is called a 
{\em $(\lambda,\epsilon)$-quasigeodesic}\index{quasigeodesic}
if every subword $w'$ of $w$ has length at most $\lambda |g'| + \epsilon$, where $g'$ is the element represented by $w'$.
Note that it is not clear that all combinations of properties can be achieved within the language of a single
automatic structure. In particular it is an open question \cite{ECHLPT} whether, given an automatic structure for a group $G$, an automatic structure can be derived for $G$ that is both prefix closed and has uniqueness.

Note that the definitions of automaticity and biautomaticity are independent 
of choice of generating set; that is if $G$ has an automatic structure
over a finite generating set $X$, then it has one over any other finite generating set $Y$.

\subsection{Basic properties of automatic groups}
\label{sec:properties}

Some properties of automatic groups can be deduced very easily from
basic properties of regular languages, which imply certain constraints
on their Cayley graphs.
In particular any automatic group is finitely presented\index{finitely presented} with soluble word problem\index{word problem},
and quadratic Dehn function\index{Dehn function}, while any biautomatic group has soluble conjugacy 
problem\index{conjugacy problem}. 
We recall that the word problem is soluble in $G$ if an algorithm exists that can decide whether or not any input word represents the identity, and the
conjugacy problem is soluble if an algorithm exists that can decide whether or not two input words represent elements that are conjugate within the group;
it is an open question whether the conjugacy problem is soluble for automatic groups.
It also is an open question whether the isomorphism problem\index{isomorphism problem}
is soluble for automatic groups, that is, whether an algorithm that was given as input automatic structures for a pair of groups $G,H$ could decide whether or not $G$ and $H$ were isomorphic. It is conjectured in \cite{ECHLPT} that this problem is insoluble. Note that it is soluble for hyperbolic groups\index{hyperbolic group} \cite{DahmaniGuirardel,Sela}.

In order to explain these statements in more detail, 
we use the language of van 
Kampen diagrams. Informally (essentially, following \cite{Johnson}), given a group
$G$ with presentation $\langle X \mid R \rangle$ and a word $w$ over $X$ that 
represents the identity of $G$, we define a van Kampen diagram\index{van Kampen diagram} $\Delta_w$ for $w$ to be a finite, connected, directed, planar graph, with a selected {\em basepoint}, whose directed edges are labelled by elements of $X$, in such a way that the boundary of every face of 
the graph (known as a {\em cell}) is labelled (from some starting point, in some orientation) by a word from $R$, while the boundary 
of the graph is labelled (from the basepoint) by $w$. As a directed, edge labelled graph, $\Delta_w$ maps (not necessarily injectively) into the Cayley 
graph $\Cay(G,X)$.
The {\em area of the diagram} $\Area(\Delta_w)$ is defined to be the number of cells it contains; of course its value is dependent on the set $R$, and would change if $R$ were changed.

We define the {\em area of the word} $w$ to be the minimum of the areas of all
van Kampen diagrams that represent $w$. 
And we define the {\em Dehn function}\index{Dehn function} 
(or {\em isoperimetric function}\index{isoperimetric function}) for $G$, $f: \N \rightarrow \N$,
to be the function for which $f(n)$ is the maximum area of all words $w$ of length $n$ over $X^\pm$ that represent
the identity of $G$.
Although the precise form of the Dehn function depends on the chosen presentation for $G$, it can be shown that two Dehn functions corresponding to different presentatives are related by a natural notion of equivalence, and in particular if one is polynomially bounded, then both are, by polynomials of the same degree.

\begin{proposition}
\label{prop:fp_wp}
Every automatic group is finitely presented, with a quadratic upper bound on the Dehn function,
and hence soluble word problem.
\end{proposition}
We sketch the proof, which is that of \cite[Theorem 5.2.13]{HRR}.
\begin{proof}
We suppose that $L,k$ are the language and fellow traveller constant of an
automatic structure over a generating set $X$; we may assume that $L$ consists of quasigeodesics.
Suppose that $w=a_1\cdots a_n$ is a word of length $n$ representing the identity.
Now we define words $w_0,\ldots ,w_n$ as follows.
We define $w_0=w_n$ to be a representative in $L$ of $1$, and for each $i=1,\ldots,n-1$ we choose $w_i$ to be a representative in $L$ of the prefix of $w$ of length $i$; 
since $L$ is quasigeodesic, we can choose $w_i$ of length 
at most $|w_0|+Ci$, for some constant $C$ of the automatic structure.
We start with a disk within the plane whose boundary is labelled by $w$, and
divide it into cells to
form a van Kampen diagram $\Delta_w$ with boundary $w$ as follows.
First, a loop labelled by $w_0$ connects the basepoint to itself, while for each $i$ a path labelled $w_i$ connects the basepoint to the point on the boundary distance $i$ along 
$w$, and none of these paths cross each other.
Then, since the paths ${}_1w_{i-1}$ and ${}_1w_i$ in $\Cay(G,X)$ fellow travel at distance at most $k$, we can construct paths of length at most $k$ that
connect corresponding vertices on the paths within the disk labelled by those two words,  
and hence divide the region between the two paths into cells each of length at most $2k+2$.
In this way we divide the interior of the diagram into 
a number of cells labelled by words of length at most $2k+2$.
together with two cells labelled by the word $w_0=w_n$, as illustrated in 
Figure~\ref{fig:quadratic_dehn}.

\begin{figure}[htb]
\setlength{\unitlength}{0.8pt}
\begin{picture}(240,240)(-80,-120)
\put(0,0){\circle*{10}}
\put(100,0){\oval(200,200)}
\put(115,100){\vector(1,0){0}} \put(120,100){\vector(1,0){0}} 
\put(120,103){\small $w$}
\put(0,18){\oval(80,30)[rb]}
\put(40,18){\vector(0,1){0}} \put(42,13){\small $w_0$}
\put(22,18){\oval(36,30)[rt]}
\put(0,0){\line(2,3){22}}
\put(0,-18){\oval(80,30)[rt]}
\put(40,-18){\vector(0,-1){0}} \put(43,-23){\small $w_n$}
\put(22,-18){\oval(36,30)[rb]}
\put(0,0){\line(2,-3){22}}
\put(0,25){\vector(0,1){0}} \put(-16,20){\small $a_1$}
\put(70,18){\vector(0,1){0}} \put(72,13){\small $w_1$}
\put(0,20){\oval(140,40)[r]}
\put(0,60){\vector(0,1){0}} \put(-16,55){\small $a_2$}
\put(90,50){\vector(0,1){0}} \put(92,45){\small $w_2$}
\put(0,40){\oval(180,80)[r]}
\put(40,100){\vector(1,0){0}} \put(40,103){\small $a_3$}
\put(110,50){\vector(0,1){0}} \put(112,45){\small $w_3$}
\put(70,50){\oval(80,100)[r]}
\put(115,60){\line(1,0){5}} \put(125,60){\line(1,0){5}} 
\put(135,60){\line(1,0){5}} \put(145,60){\line(1,0){5}} 
\put(115,20){\line(1,0){5}} \put(125,20){\line(1,0){5}} 
\put(135,20){\line(1,0){5}} \put(145,20){\line(1,0){5}} 
\put(75,-20){\line(1,0){5}} \put(85,-20){\line(1,0){5}} 
\put(95,-20){\line(1,0){5}} \put(105,-20){\line(1,0){5}} 
\put(115,-20){\line(1,0){5}} \put(125,-20){\line(1,0){5}} 
\put(135,-20){\line(1,0){5}} \put(145,-20){\line(1,0){5}} 
\put(115,-30){\oval(90,60)[b]}
\put(35,-30){\oval(70,60)[tr]}
\put(160,-30){\line(0,1){130}}
\put(200,40){\vector(0,-1){0}} \put(205,40){\small $a_{n-4}$}
\put(115,-40){\oval(130,60)[b]}
\put(25,-40){\oval(50,80)[tr]}
\put(200,-40){\oval(40,90)[tl]}
\put(150,-40){\oval(200,80)[bl]}
\put(150,-100){\oval(60,40)[tr]}
\put(160,-20){\vector(0,1){0}} 
\put(180,-20){\vector(0,1){0}} 
\put(100,-40){\oval(100,100)[bl]}
\put(80,-100){\oval(60,20)[tr]}
\put(200,-50){\vector(0,-1){0}} \put(205,-50){\small $a_{n-3}$}
\put(160,-80){\vector(1,0){0}} \put(158,-77){\small $w_{n-3}$}
\put(140,-100){\vector(-1,0){0}} \put(140,-110){\small $a_{n-2}$}
\put(100,-90){\vector(1,0){0}} \put(98,-87){\small $w_{n-2}$}
\put(50,-100){\vector(-1,0){0}} \put(50,-110){\small $a_{n-1}$}
\put(8,-50){\oval(84,80)[br]}
\put(15,-90){\vector(-1,0){0}}
\put(15,-87){\small $w_{n-1}$}
\put(0,-45){\vector(0,1){0}}\put(-16,-50){\small $a_n$}
{\color{blue}
\put(17,25.5){\line(0,1){14.5}} \put(17,40){\line(3,2){57}}
\put(35,30){\line(0,1){10}} \put(35,40){\line(3,1){55}}
\put(40,55){\small $\leq k$}
\put(41,18){\line(2,1){26}} \put(67,31){\line(1,0){43}}
\put(38,8){\line(1,0){64}}
\put(90,58){\line(1,0){20}}
\put(74,78){\line(3,1){30}}
\put(0,40){\line(1,1){40}}
\put(40,80){\line(2,1){40}}
\put(0,80){\line(4,1){80}}
\put(38,-8){\line(1,0){24}}
\put(36,-30){\line(1,0){34}}
\put(50,-48){\line(3,1){20}}
\put(26,-33){\line(2,-3){22}}
\put(22,-33){\line(0,-1){57}}
\put(10,-15){\line(0,-1){75}}
\put(20,-90){\line(1,0){60}}
\put(80,-55){\line(0,-1){35}}
\put(110,-60){\line(0,-1){40}}
\put(140,-60){\line(0,-1){20}}
\put(140,-80){\line(-3,-2){30}}
\put(160,-25){\line(1,-1){20}} \put(180,-45){\line(0,-1){55}}
\put(160,10){\line(1,-1){20}} 
\put(160,45){\line(1,-1){40}}
\put(170,55){\small $\leq k$}
\put(160,85){\line(1,-2){40}}
}
\end{picture}
\caption{Van Kampen diagram for a representative of the identity in an automatic group}
\label{fig:quadratic_dehn}
\end{figure}
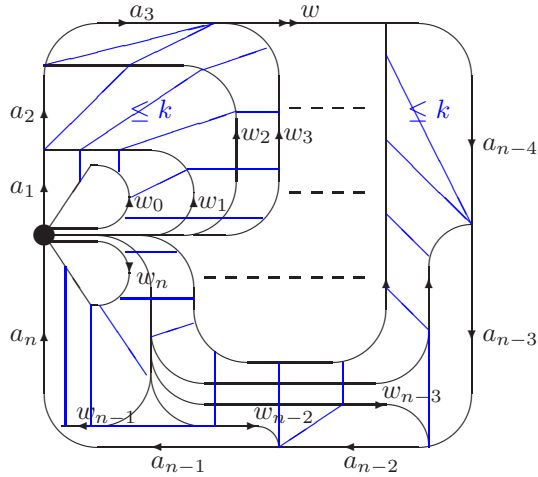

Using the bounds on $|w_i|$, we see that the total number of cells is bounded by a quadratic function of $n$.
We now define $R$ to be the set of all words of length up to $2k+2$ that represent the identity, together with the word $w_0$. Then $\langle X \mid R \rangle$ is a finite presentation for $G$,
and, relative to $R$, $\Delta_w$ has quadratic area.
\end{proof}
A similar argument proves an exponential upper bound on the Dehn function
for any asynchronously automatic group; it is an open question \cite{ECHLPT} whether a polynomial time solution to the word problem must exist. 

The most straightforward way to prove a group non-automatic is probably to 
show that it has a Dehn function that is above quadratic. 
This argument proves easily the non-automaticity of the Baumslag--Solitar groups
$\langle a, b \mid ba^pb^{-1} = a^q$ for which $p,q>0$ and $p \neq q$, 
since they have exponential Dehn function; in fact they provide 
examples of non-automatic groups that are asynchronously automatic.

But there are many groups with quadratic Dehn functions that are known by other methods  not to be automatic.

The non-automaticity of
the groups $\SL_n(\Z)$ for $n \geq 3$ is proved in \cite{ECHLPT}.  
The group $\SL_2(\Z)$ is well known to be virtually free, and hence hyperbolic\index{hyperbolic group} with linear Dehn function. The group $\SL_3(\Z)$ has exponential Dehn function 
and so is certainly non-automatic.
However, the existence of a quadratic Dehn function for $\SL_n(\Z)$ with $n \geq 5$ was proved in \cite{Young} in 2013 (and had been conjectured by Thurston, in fact for $n \geq 4$).
In order to prove non-automaticity of the group for all $n \geq 3$, Epstein and Thurston derived higher dimensional isoperimetric inequalities\index{isoperimetric function!higher dimensional} that would have 
to hold in any combable group\index{combable group} of isometries acting properly discontinuously 
with compact quotient on a $k$-connected Riemannian manifold \cite[Theorem 10.3.5]{ECHLPT}.
The non-automaticity of $\SL_n(\Z)$ now follows by the construction
of a proper discontinous cocompact action on a suitable contractible manifold,
and the demonstration that a higher dimensional isoperimetric inequality fails; 
hence $\SL_n(\Z)$ is proved to be non-combable and so non-automatic. 

Van Kampen diagrams can also be used to prove solubility of the conjugacy 
problem\index{conjugacy problem} in any biautomatic group, by demonstrating the existence
of a conjugator of bounded length.
The proof below, valid for any bicombable group,
is taken from
\cite{Short90}; an earlier result of \cite{GerstenShort91} constructs an automaton out of the biautomatic structure to solve
the problem.

\begin{proposition}
\label{prop:cp}
Given a biautomatic group $G$, any two words $u,v$ representing
conjugate elements are conjugate by an element of length at most $a^{|u|+|v|}$,
	for some constant $a$ (depending only on the biautomatic structure).
Hence any biautomatic group has soluble conjugacy problem.
\end{proposition}
\begin{proof}
We choose a  biautomatic structure $(L,k)$ over a finite generating set $X$, 
and suppose that the
words $u,v$ over $X^\pm$ represent conjugate elements of $G$.
	Let $N:=|X^\pm|^{k(|u|+|v|)}$. We find a conjugator of length at most $N$, and so $a=|X^\pm|^k$.

For suppose that an element $g \in G$ conjugates $u$ to $v$, that is that $gu=_G vg$, and
that $w,w' \in L$ represent the elements $g$ and $ug$, respectively.
	We consider the paths ${}_1w$, ${}_1w'$ and ${}_uw$ within the Cayley graph $\Cay(G,X)$, and see that the biautomaticity of $G$  ensures that ${}_1w$ 
	and ${}_1w'$ 
	fellow travel at distance at most $|u|k$, and that ${}_1w'$ and ${}_uw$
fellow travel at distance at most $|v|k$.
We deduce that we can construct a van Kampen diagram with boundary 
labelled by $wuw^{-1}v^{-1}$ in which chords of length at most $(|u|+|v|)k$ join boundary vertices
in corresponding positions on the two boundary subwords labelled by $w$,
as shown on the left hand side of Figure~\ref{fig:biaut_cp}. 
Where $|w|=n$, let $d_1,d_2,\ldots,d_{n-1}$ be the words that
label those chords.


\begin{figure}[htb]
\begin{minipage}[l]{2.2in}
\begin{picture}(200,210)
\put(10,10){\circle*{3}}
\qbezier(10,10)(-10,50)(10,100)
\put(-3,90){$u$}
\put(10,100){\vector(1,2){0}}
\qbezier(10,100)(30,150)(10,190)
\put(10,190){\circle*{3}}
\qbezier(10,190)(50,210)(100,190)
\put(110,186){\vector(2,-1){0}}
\put(100,193){$w$}
\qbezier(100,190)(150,170)(190,190)
\put(190,190){\circle*{3}}
\qbezier(190,100)(210,150)(190,190)
\put(190,100){\vector(1,2){0}}
\put(190,90){$v$}
\qbezier(190,10)(170,50)(190,100)
\put(190,10){\circle*{3}}
\qbezier(10,10)(50,30)(100,10)
\put(110,6){\vector(2,-1){0}}
\put(95,0){$w$}
\qbezier(100,10)(150,-10)(190,10)
\put(5,5){\makebox(0,0){$1$}}
\put(10,10){\line(1,1){180}}
\put(110,110){\vector(1,1){0}}
\put(110,100){$w'$}
\put(30,18){\circle*{3}} \put(30,198){\circle*{3}}
\put(26,26){\circle*{3}}
\qbezier(30,18)(10,60)(30,110) \qbezier(30,110)(50,150)(30,198)
\put(50,20){\circle*{3}} \put(50,200){\circle*{3}}
\put(43,43){\circle*{3}}
\qbezier(50,20)(30,80)(50,120) \qbezier(50,120)(70,155)(50,198)
\put(85,16){\circle*{3}} \put(85,196){\circle*{3}}
\put(86,86){\circle*{3}}
\qbezier(85,85)(105,135)(85,196) \qbezier(85,17)(75,55)(85,85)
\put(80,56){\vector(0,1){0}} 
\put(70,46){$d_i$} 
\put(130,1){\circle*{3}} \put(130,181){\circle*{3}}
\put(139,139){\circle*{3}}
\qbezier(130,70)(150,120)(130,181) \qbezier(130,2)(120,40)(130,70)
\put(125,40){\vector(0,1){0}} 
\put(115,30){$d_j$} 
\put(150,1){\circle*{3}} \put(150,180){\circle*{3}}
\put(157,157){\circle*{3}}
\qbezier(150,80)(170,130)(150,180) \qbezier(150,1)(140,50)(150,80)
\put(170,3){\circle*{3}} \put(170,183){\circle*{3}}
\put(173,173){\circle*{3}}
\qbezier(170,90)(190,130)(170,183)
\qbezier(170,3)(150,50)(170,90)
\end{picture}
\end{minipage}
\begin{minipage}[r]{2.2in}
\begin{picture}(200,200)(-60,0)
\put(10,10){\circle*{3}}
\qbezier(10,10)(-10,50)(10,100)
\put(-3,90){$u$}
\put(10,100){\vector(1,2){0}}
\qbezier(10,100)(30,150)(10,190)
\put(10,190){\circle*{3}}
\qbezier(10,190)(50,210)(85,196)
\put(75,199){\vector(2,-1){0}}
\put(65,205){$\hat{w}$}
\qbezier(85,196)(115,186)(145,205)
\put(145,205){\circle*{3}}
\qbezier(145,115)(165,165)(145,205)
\put(145,115){\vector(1,2){0}}
\put(145,104){$v$}
\qbezier(145,25)(125,65)(145,115)
\put(145,25){\circle*{3}}
\qbezier(10,10)(50,30)(85,16)
\put(75,19){\vector(2,-1){0}}
\put(62,10){$\hat{w}$}
\qbezier(85,16)(115,14)(145,25)
\put(5,5){\makebox(0,0){$1$}}
\put(30,18){\circle*{3}} \put(30,198){\circle*{3}}
\qbezier(30,18)(10,60)(30,110) \qbezier(30,110)(50,150)(30,198)
\put(50,22){\circle*{3}} \put(50,202){\circle*{3}}
\qbezier(50,22)(30,80)(50,120) \qbezier(50,120)(70,155)(50,202)
\put(85,16){\circle*{3}} \put(85,196){\circle*{3}}
\qbezier(85,85)(105,135)(85,196) \qbezier(85,17)(75,55)(85,85)
\put(80,50){\vector(0,1){0}} 
\put(65,35){$d_i = d_j$} 
\put(105,16){\circle*{3}} \put(105,193){\circle*{3}}
\qbezier(105,95)(125,145)(105,193) \qbezier(105,16)(95,65)(105,95)
\put(125,18){\circle*{3}} \put(125,196){\circle*{3}}
\qbezier(125,105)(145,145)(125,196)
\qbezier(125,18)(105,65)(125,105)
\end{picture}
\end{minipage}
\caption{Finding a conjugator of bounded length}
\label{fig:biaut_cp}
\end{figure}
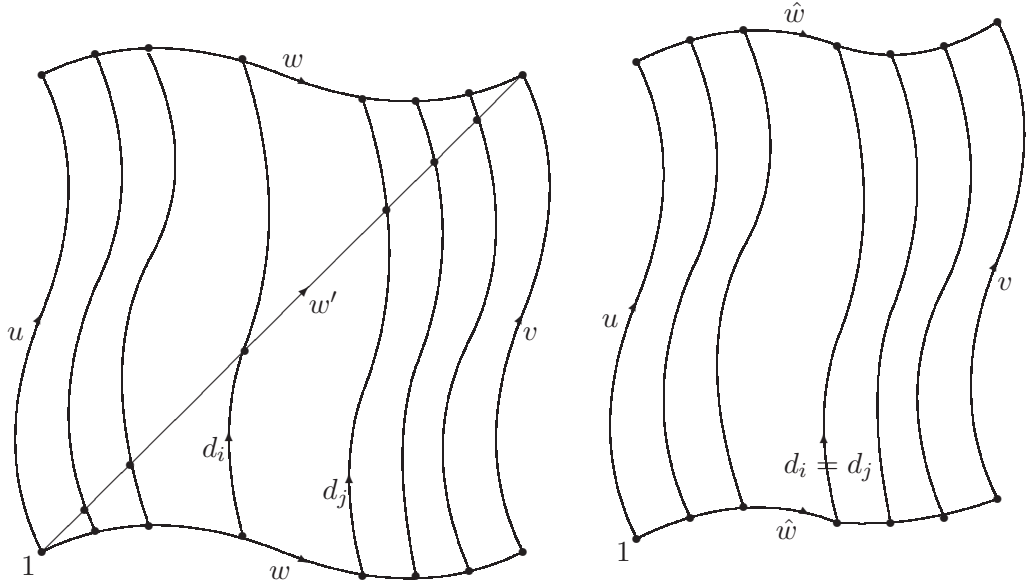

Now if $n>|X^\pm|^{(|u|+|v|)k}$, then for some $i,j$ we have $d_i=d_j$. In 
that case, where $\hat{w}$ is the word formed from $w$ by deleting its
middle section of length $j-i$, from its $(i+1)$-th to its $j$-th letter, 
we can form the van Kampen diagram with boundary word $\hat{w}u\hat{w}^{-1}v^{-1}$
shown on the right hand side of Figure~\ref{fig:biaut_cp} 
by deleting the central part of the diagram we already constructed for $wuw^{-1}v^{-1}$. 

\end{proof}


Various combinations of automatic groups are known to be automatic \cite{ECHLPT,BGSS}: these include
free products, direct products, certain amalgamated products and HNN extensions of automatic groups,
as well as subgroups of finite index in automatic groups, groups with
automatic groups as subgroups of finite index, quotients of automatic
groups by finite normal subgroups. Some, but not all, of these closure properties
also hold for biautomatic groups.
It is an open question whether direct factors of automatic groups must be automatic
(but the analogous result is proved for biautomatic groups \cite{Mosher97}).
It is also open \cite{ECHLPT} whether a group with a biautomatic group as a subgroup of
finite index must be biautomatic.

\subsection{Basic examples and non-examples}

\subsubsection{Virtually abelian groups, soluble groups}
\label{sec:soluble}

We already described shortlex automatic structures for the free abelian group 
$\Z^n$. In fact $\Z^n$ is also biautomatic, but with a different (less 
straightforward) language, and indeed so is every virtually 
abelian group. However it was already proved in \cite{ECHLPT} that an
automatic nilpotent group must be virtually abelian; the proof uses the fact that a regular language with polynomial growth cannot satisfy a (synchronous) fellow traveller property.
It was conjectured by Thurston that the same result must hold for an automatic
soluble group. That conjecture remains open, but it was proved
for automatic polycylic groups in \cite{Harkins}, using an embedding of a 
finite index subgroup of a polycyclic group of exponential growth as a lattice 
in an appropriate Lie group, where \cite[Theorem 10.3.5]{ECHLPT} about higher dimensional isoperimetric functions\index{isoperimetric function!higher dimensional} could be applied, 
which had previously been used to prove the non-automaticity of $\SL_n(\Z)$ for $n \geq 3$. 
Much more recently it was proved in \cite{Romankov} that biautomatic soluble groups must be virtually abelian.

\subsubsection{Hyperbolic groups}
\label{sec:hyperbolic}

Maybe the most natural examples of non-abelian automatic groups are provided by
the large family of {\em word hyperbolic} groups, which contains all finitely
generated free groups as well as the fundamental groups of all compact hyperbolic manifolds. 

A group $G$ with finite 
generating set $X$ is said to be {\em word hyperbolic}\index{hyperbolic group} if its Cayley
graph $\Cay(G,X)$ is a $\delta$-hyperbolic metric space, for some $\delta \geq 0$;
a geodesic metric space $(\cX,d)$ is $\delta$-hyperbolic
if for any triangle in $\cX$ with geodesic sides $\gamma_1,\gamma_2,\gamma_3$ and 
for any vertex $p$ on the side $\gamma_1$ there is a vertex $q$ on the 
union $\gamma_2 \cup \gamma_3$ of the other two sides for which
$d(p,q) < \delta$ (we say that triangles in $\cX$ are {\em $\delta$-slim}).
The property of being word hyperbolic is
independent of the choice of a finite generating set for $G$, although the value of $\delta$ is not.
The fundamental groups of compact hyperbolic manifolds give examples, as
do finitely generated free groups (which are $0$-hyperbolic with respect to free generating sets).

We note that there are many equivalent definitions of hyperbolicity for metric spaces (and hence for finitely generated groups), which are explained in \cite{AlonsoEtAl}. In particular there is a characterisation in terms of {\em thin} rather than {\em slim triangles} (and a linear relationship between the associated
parameters ``$\delta$'').

It is proved in \cite{ECHLPT} that a word hyperbolic group\index{hyperbolic group} $G$ is automatic
over any generating set $X$,
with an automatic structure whose language consists of all geodesic words over
the selected generating set.
The regularity of that set of geodesic words is equivalent
to the fact that the Cayley graph $\Cay=\Cay(G,X)$ contains finitely many
cone types.
For $g \in G$, represented by a geodesic word $w$,
we define the {\em cone} $C(g)$ (or $C(w)$) on the vertex $g$ of $\Cay$ to be the set of (geodesic) paths 
$\gamma$ within $\Cay$ starting at $g$ for which the concatenation $\eta\gamma$
of a geodesic path $\eta$ from $1$ to $g$ with $\gamma$ is also geodesic. The {\em cone type} $[C(g)]$ or $[C(w)]$ of the cone is defined to be the set of words that label the paths within it.
Now for $y \in X \cup X^{-1}$, if $wy$ is also geodesic then for any word $v$,
\[ v \in [C(wy)] \iff yv \in [C(w)]. \]
It follows that we can recognise the set of geodesic words over $X^\pm$ with an \fsa
whose states correspond to the cone types, with a transition from
$[C(w)]$ to $[C(wy)]$ on $y$ whenever $wy$ is geodesic, but otherwise to a single {\em failure state} (i.e. a non-accepting sink state).
We can illustrate this construction in the free abelian group $\Z^2$
with generating set $\{a,b\}$, where there are nine cone types $[C(w)]$,
defined by the nine geodesic words $\emptyword$, $a$, $b$, $a^{-1}$, $b^{-1}$, $ab$, $ab^{-1}$, $a^{-1}b$, $a^{-1}b^{-1}$,
and consisting of the nine possible sets of geodesic words in which each generator appears either only with positive exponent, or only with negative exponent, or not at all.
The \fsa is illustrated in Figure~\ref{fig:geofsa_Z2}. This automaton
is {\em not} part of an automatic structure for $\Z^2$; it cannot be 
since, for example, the vertices distance $i$ from the origin on the geodesic words $a^ib^i$ and $b^ia^i$ are distance $2i$ apart within the Cayley graph, and hence
this language does not satisfy a fellow travelling property.

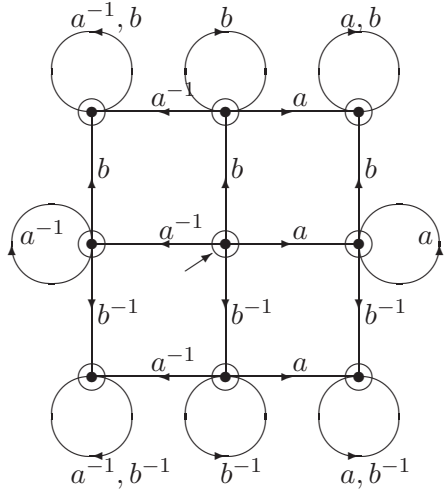
\begin{figure}[h]
\begin{picture}(400,180)(-200,-90)
\put(-15,-10){\vector(3,2){10}}
\put(0,0){\circle*{4}}
\put(0,0){\circle{10}}
\put(50,0){\circle*{4}}
\put(50,0){\circle{10}}
\put(-50,0){\circle*{4}}
\put(-50,0){\circle{10}}
\put(0,50){\circle*{4}}
\put(0,50){\circle{10}}
\put(0,-50){\circle*{4}}
\put(0,-50){\circle{10}}
\put(0,0){\line(1,0){50}}
\put(25,0){\vector(1,0){0}}
\put(25,2){$a$}
\put(0,0){\line(-1,0){50}}
\put(-25,0){\vector(-1,0){0}}
\put(-25,2){$a^{-1}$}
\put(0,0){\line(0,1){50}}
\put(0,25){\vector(0,1){0}}
\put(2,25){$b$}
\put(0,0){\line(0,-1){50}}
\put(0,-25){\vector(0,-1){0}}
\put(2,-30){$b^{-1}$}
\put(65,0){\oval(30,30)}
\put(80,0){\vector(0,1){0}}
\put(72,0){$a$}
\put(-65,0){\oval(30,30)}
\put(-80,0){\vector(0,1){0}}
\put(-77,0){$a^{-1}$}
\put(0,65){\oval(30,30)}
\put(0,80){\vector(1,0){0}}
\put(-2,82){$b$}
\put(0,-65){\oval(30,30)}
\put(0,-80){\vector(1,0){0}}
\put(-2,-90){$b^{-1}$}
\put(0,50){\line(1,0){50}}
\put(25,50){\vector(1,0){0}}
\put(25,52){$a$}
\put(50,50){\circle*{4}}
\put(50,50){\circle{10}}
\put(50,65){\oval(30,30)}
\put(50,80){\vector(1,0){0}}
\put(43,82){$a,b$}
\put(50,0){\line(0,1){50}}
\put(50,25){\vector(0,1){0}}
\put(52,25){$b$}
\put(0,-50){\line(1,0){50}}
\put(25,-50){\vector(1,0){0}}
\put(25,-48){$a$}
\put(50,-50){\circle*{4}}
\put(50,-50){\circle{10}}
\put(50,-65){\oval(30,30)}
\put(50,-80){\vector(1,0){0}}
\put(43,-90){$a,b^{-1}$}
\put(50,0){\line(0,-1){50}}
\put(50,-25){\vector(0,-1){0}}
\put(52,-30){$b^{-1}$}
\put(0,50){\line(-1,0){50}}
\put(-25,50){\vector(-1,0){0}}
\put(-28,52){$a^{-1}$}
\put(-50,50){\circle*{4}}
\put(-50,50){\circle{10}}
\put(-50,65){\oval(30,30)}
\put(-50,80){\vector(-1,0){0}}
\put(-58,82){$a^{-1},b$}
\put(-50,0){\line(0,1){50}}
\put(-50,25){\vector(0,1){0}}
\put(-48,25){$b$}
\put(0,-50){\line(-1,0){50}}
\put(-25,-50){\vector(-1,0){0}}
\put(-28,-48){$a^{-1}$}
\put(-50,-50){\circle*{4}}
\put(-50,-50){\circle{10}}
\put(-50,-65){\oval(30,30)}
\put(-50,-80){\vector(-1,0){0}}
\put(-58,-90){$a^{-1},b^{-1}$}
\put(-50,0){\line(0,-1){50}}
\put(-50,-25){\vector(0,-1){0}}
\put(-48,-30){$b^{-1}$}
\end{picture}
\caption{\fsa recognising geodesics in $\Z^2$}
\label{fig:geofsa_Z2}
\end{figure}

Given the finiteness of the set of cone types in a word hyperbolic group\index{hyperbolic group},
biautomaticity of any word hyperbolic group now follows
once it is observed that the fellow travelling of two geodesic words 
with common (or adjacent) start and end vertices can be derived from the
slimness of triangles.
In fact, it is proved by Papasoglu \cite{Papasoglu} that this fellow traveller condition
characterises word hyperbolic groups,
and hence so does the existence of a (bi)automatic structure that consists
of all geodesic words.
A procedure to test for hyperbolicity that is based on this result is 
described in \cite{Wakefield}.
Starting with a shortlex automatic structure $(L,k)$ for a group $G$ over $X$, 
the procedure attempts to construct an automatic structure $(\widehat{L},\widehat{k})$ with $\widehat{L} \supset L$ and $\widehat{k}\geq k$, and such that
$\widehat{L}$ contains all geodesic words over $X^\pm$. 
It will terminate with such a structure precisely when $G$ is hyperbolic.
An improved procedure, based on the same result was developed by Holt and Epstein \cite{EpsteinHolt01} and implemented in \kbmag.

The fundamental groups of finite volume hyperbolic manifolds (geometrically 
finite hyperbolic groups\index{hyperbolic group!geometrically finite}) were proved biautomatic by Epstein \cite{ECHLPT}, 
with a further biautomatic structure subsequently described by Lang 
\cite{Lang}. 

Geometrically finite hyperbolic groups were the motivating examples for Bowditch's  
definition \cite{Bowditch} of  a group hyperbolic relative to a collection of subgroups;
a geometrically finite hyperbolic group is hyperbolic relative to a collection of abelian groups\index{hyperbolic group!relatively}. The major part of the definition of relative hyperbolicity 
is the requirement that the Cayley graph of a group 
hyperbolic relative to a collection $\cH$ of subgroups becomes hyperbolic 
after the contraction of edges within left cosets of subgroups in $\cH$. However weaker and stronger versions of the definition exist depending on whether or not a 
condition of {\em bounded coset penetration} is required to hold. Under the
stronger definition (studied in \cite{Osin}) it is proved, in particular 
in \cite{AntolinCiobanu}, that groups hyperbolic relative to shortlex 
biautomatic 
subgroups are themselves shortlex biautomatic. The shortlex biautomaticity of geometrically finite hyperbolic groups is a consequence of this result.

A further generalisation of hyperbolic groups is provided by
semihyperbolic groups\index{hyperbolic group!semihyperbolic}\index{semihyperbolic group}, which were introduced 
by Bridson and Alonso in \cite{AlonsoBridson}; the class contains all biautomatic groups (hence all hyperbolic groups\index{hyperbolic group}) and all CAT(0) groups\index{CAT(0)!group} (see Section~\ref{sec:actions}).
A group $G$ with finite generating set $X$ is defined to be {\em weakly semihyperbolic} if $\Cay(G,X)$ admits a bounded quasi-geodesic bicombing 
(with a unique combing path $s_{g_1,g_2}(t)$ identified between any pair $g_1,g_2$ of vertices of the graph), and  {\em semihyperbolic} if it has such a bicombing that is equivariant under the action of $G$ (so that 
$g.s_{g_1,g_2}(g)=s_{gg_1,gg_2}(t)$). This class of groups satisfies many closure properties, and all groups within it are finitely presented, with soluble word\index{word problem} and conjugacy problems\index{conjugacy problem}.

\subsubsection{Fundamental groups of compact 3-manifolds}

It is proved in \cite{ECHLPT} that the fundamental groups of
compact 3-manifolds\index{3-manifold} based on
six of Thurston's eight model geometries for compact 3-manifolds \cite{Thurston3DGT} admit automatic structures. But it is also proved that the fundamental groups of closed manifolds based on
the \Nil and \Sol geometries (which are non-abelian, nilpotent and soluble, respectively) cannot even be asynchronously automatic \cite{ECHLPT,Brady93}.

However, using combination theorems for automatic groups, it can be proved 
(as in \cite[Theorem 12.4.7]{ECHLPT}, but our wording is slightly different) that an orientable,
connected, compact 3-manifold with incompressible toral boundary
whose prime factors have JSJ decompositions containing only hyperbolic pieces
has automatic fundamental group.
It was proved in \cite[Theorem B]{BridsonGilman} that the fundamental group of a manifold as above in which manifolds based on \Nil and \Sol are allowed within the JSJ 
decomposition, while not automatic, still admits an asynchronous combing 
based on an indexed language\index{language!indexed}
\cite{Aho}.

\section{Computing with automatic groups}
\label{sec:computing}
\subsection{Building automatic structures}
The original motivation for the definition of automatic groups was
computational, and so it was important from the beginning of the subject 
to be able
to construct automatic structures, that is, given a presentation for a group $G$,
to have a mechanism for building the word acceptor\index{word acceptor} and multiplier automata\index{multiplier automaton} of
an associated automatic structure. 
Software to build these automata was developed at the University of Warwick,
and the procedure used is described in \cite{EHR}. The original programs were subsequently rewritten by Holt, and released within his \kbmag package \cite{kbmag},
now available within both \gap and \magma computational systems \cite{gap,magma}.

The basic procedure is the same in both versions (the ideas are due to Holt)
and we describe it briefly now, but refer the reader to \cite{EHR} or
\cite{HoltKbmag} for more details.

A presentation for a group $G$ over a finite generating set $X$ is input, 
together with an ordering of the set $X^\pm$. The procedure attempts to prove
$G$ to be shortlex automatic over $X$ (with the given ordering) by first 
constructing a set of automata consisting of $W$ and $M_y$ for $y \in X^\pm \cup \{\emptyword\}$, and then
attempting to verify that those automata are indeed the automata 
of a shortlex automatic structure. If verification tests fail, some looping
is possible within the procedure, and indeed that looping could continue
indefinitely (or at least until the computer runs out of resources). 
If all verification tests pass, then the procedure will have
verified the shortlex automaticity of $G$ by construction and checking of a shortlex 
automatic structure.

So the procedure may succeed in proving shortlex automaticity of $G$.
But if it fails, it has certainly not proved that $G$ is not automatic,
or even that $G$ is not shortlex automatic, but rather it suggests that 
$G$ is unlikely to be shortlex automatic over the given generating set $X$,
with the given ordering of the elements of $X$.
We note that the question of automaticity for a finitely presented group is 
undecidable in general; this follows from the undecidability of questions such as triviality for a group.
We note too that it is an open question \cite{ECHLPT} whether every automatic group must be shortlex automatic with respect to some ordered generating set.

The first step of the procedure to prove shortlex automaticity is the 
construction of a rewrite
system  $\cR$ from the group presentation that is compatible with the shortlex
order. By definition, $\cR$ is a set of substitution rules 
$\rho: u \rightarrow v$, for $u,v \in (X^\pm)^*$, and with $v \slexle u$;
in order that $\cR$ encodes the presentation we require that
every relator from the
group presentation is a cyclic conjugate of the product $uv^{-1}$  or its inverse for at least one such rule.

The next step is to run the Knuth--Bendix procedure for a while on $\cR$.
The Knuth--Bendix procedure (described in \cite{HRR}) is a general procedure that,
given as input a rewrite system $\cR$ for strings compatible with a partial order, modifies
it by adding rules that are consequences of existing rules and deleting rules that 
have become redundant, in order to produce a new rewrite system.
The procedure attempts to build a finite {\em complete} system,
for which any input word $w$ can be rewritten 
after a finite number of steps to a unique irreducible word $w'$ (where 
irreducible means that $w'$ cannot be rewritten further).
However with this goal the procedure may never terminate; all that is guaranteed is
that after bounded time the modified system must
contain enough rules to reduce any word up to some bounded length to an irreducible.

In fact the procedure to construct a shortlex automatic structure
for $G$ does not need the Knuth--Bendix procedure to terminate on the input
rewrite system $\cR$. Instead, while the Knuth--Bendix procedure is running it
accumulates the set $\cD$ of {\em word differences}\index{word difference} $u(i)^{-1}v(i)$ and their
inverses (reduced according to the current modification of $\cR$) 
that correspond to prefixes of the rules $u \rightarrow v$ in the system.
Where $u=u_1\cdots u_m$, and $v=v_1\cdots v_{m'}$, a transition is added from
each word difference $u(i)^{-1}v(i)$ to $u(i+1)^{-1}v(i+1)$, creating a 
{\em word difference machine}\index{word difference!machine} that can recognise fellow travelling with respect to $\cD$. 

The Knuth--Bendix procedure is paused when it seems that the set $\cD$
and the associated automaton have stabilised. And then a candidate word 
acceptor $\WA$\index{word acceptor} is constructed, designed to reject a word $u$ if a string $v$ exists with $v <_{\rm slex} u$ for which
$(u,v)$ fellow travels according to $\cD$ while also the word difference 
$u^{-1}v$ reduces, according to the current rewrite system, to the empty word.

Similarly, multiplier automata\index{multiplier automaton} are constructed for each $y \in X^\pm \cup \{\emptyword\}$, using a  direct product construction on automata to recognise 
pairs of words $u,v$ for which $u,v \in L(W)$, $(u,v)$ fellow travels according to $\cD$,
while also the word difference 
$u^{-1}v$ reduces, according to the current rewrite system, to $y$.

Now a series of elementary tests is applied to the candidate automata. If some of these tests fail, then $\cD$ has been proved to be inadequate, and the Knuth--Bendix
procedure is restarted. If and when those tests are passed, further tests 
known as {\em axiom checking} are applied, and a positive result for these tests proves
the automata to provide a shortlex automatic structure for $G$.
If the axiom checks fail then the procedure is abandoned.

\subsection{Calculation using the automatic structure}

Once an automatic structure has been constructed for a group $G$,
much can be computed using the automata of that structure.
Various of these functions are available within the \kbmag package \cite{kbmag}.

It is straightforward to enumerate the language of a finite state automaton.
Hence we can enumerate a set of representative words for an automatic group, with unique representation if necessary (recall that once an automatic structure has been derived, a
structure with unique representation can be derived from that).

For any regular language $L$ the generating function $\sum_{n=0}^\infty s_L(n)x^n$, where $s_L(n)$ denotes the number of words of length $n$
in $L$, is a rational function, and can be computed from an automaton recognising $L$. Hence the growth series of an automatic group is computable, given a geodesic automatic structure.

Reduction of an input word to the ``normal form'' defined by the language $L$
of the automatic structure for $G$ can be performed using a combination of
the word acceptor\index{word acceptor} and multiplier automata\index{multiplier automaton}, or alternatively using the word difference machine\index{word difference!machine}.

Finiteness of otherwise of an automatic group is immediately recognisable from
a word acceptor for an automatic structure; the language is infinite precisely
when the automaton admits loops.
In this way, the Heineken group
$G=\langle x,y,z \mid [x,[x,y]]=z,\,[y,[y,z]]=x,\,[z,[z,x]]=y \rangle$
was proved infinite, by Holt using \kbmag; computation with the automatic structure 
subsequently revealed the group to be hyperbolic. Previously that group had been proposed as a possible example of a finite group with a balanced presentation.
Similarly, a second proof of the infiniteness of the Fibonacci group $F(2,9)$ was provided by the construction of an automatic structure for it
\cite{HoltF29}.

Tests for hyperbolicity \cite{Wakefield,EpsteinHolt01} that make use of 
automatic structures for $G$ together with Papasoglu's characterisation of hyperbolic groups have already been described in Section~\ref{sec:hyperbolic}.
The second of those is implemented in \kbmag, as is an algorithm \cite{EpsteinHolt01} estimating the thinness constant (related to, but not equal to, the slimness constant) for geodesic triangles in the Cayley graph of a word hyperbolic group.

Quadratic and linear time solutions to the conjugacy problem\index{conjugacy problem} in a hyperbolic group are described in \cite{BridsonHowie} and \cite{EpsteinHolt06,BuckleyHolt}.
A practical cubic time solution that restricts to infinite order elements
is due to Marshall \cite{Marshall}, using some ideas from Swenson, and has been implemented in
the \gap system.

\section{Group actions and negative curvature}
\label{sec:actions}

One of the basic principles of
geometric group theory is generally referred to as the
the Milnor--\u{S}varc lemma:

If a group $G$ has a ``nice'' (properly discontinuous and cocompact) discrete, 
isometric action on a metric space $\cX$
then its Cayley graph is quasi-isometric to $\cX$.
In particular a group with such an action on a $\delta$-hyperbolic space is  word hyperbolic.

A variety of results derive automaticity or biautomaticity of
a group from its ``nice'' actions on spaces in which some kind of non-positive curvature can be found.

\begin{theorem}[Gersten, Short, 1990, 1991 \cite{GerstenShort90,GerstenShort91}]
\label{thm:GSaction}
A group acting discretely and fixed point freely on
a piecewise Euclidean 2-complex of type $A_1 \times A_1$, $A_2$, $B_2$ or $G_2$
(corresponding to tesselations of the Euclidean plane by squares, equilateral
triangles, or triangles with angles $(\pi/2,\pi/4,\pi/4)$ or $(\pi/2,\pi/3,\pi/6)$) is biautomatic.
\end{theorem}

As a consequence of the above results, and within the same two articles,
Gersten and Short deduce 
that groups satisfying any of the small cancellation\index{small cancellation} 
conditions C(7) or else T(p) and T(q) with $(p,q) \in \{(3,7),\,(4,5),\,(5,4)\}$
(defined in \cite{LyndonSchupp})
are hyperbolic, and hence in particular biautomatic,
and then that groups satisfying
the small cancellation conditions C(6), or C(4) and T(4), or C(3) and T(6)
are biautomatic.

A geodesic metric space $\cX$ is defined to be CAT(0)\index{CAT(0)!space} if
for any geodesic triangle in the space, and for any two points $p,q$ on the sides
of that triangle, the distance between $p$ and $q$ in $\cX$ is no more than
the distance between the points in corresponding positions on the
sides of a geodesic triangle with the same side lengths in the Euclidean plane, as illustrated in Figure~\ref{fig:cat0}.
A complete CAT(0) space is often called a {\em Hadamard space}.
A group is called CAT(0)\index{CAT(0)!group} if it acts properly and cocompactly on a CAT(0) space.

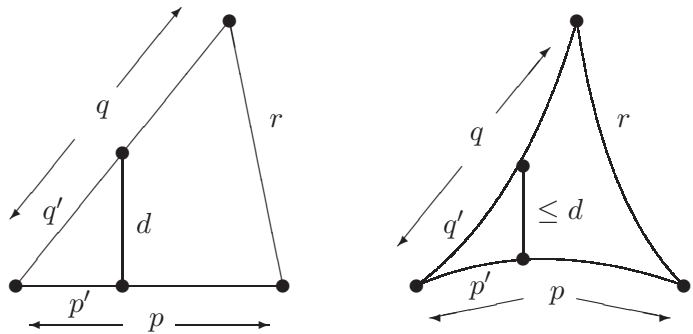
\begin{figure}[h]
	\begin{picture}(300,130)(-50,-20)
	\put(0,0){\circle*{5}}
	\put(100,0){\circle*{5}}
	\put(80,100){\circle*{5}}
	\put(0,0){\line(1,0){100}} \put(50,-15){$p$}
		\put(40,-15){\vector(-1,0){35}}
		\put(60,-15){\vector(1,0){35}}
	\put(0,0){\line(4,5){80}} \put(30,65){$q$}
		\put(22,55){\line(-4,-5){24}}\put(-2,25){\vector(-1,-1){0}}
		\put(38,75){\line(4,5){24}} \put(62,105){\vector(1,1){0}}
	\put(100,0){\line(-1,5){20}} \put(95,60){$r$}
	\put(40,0){\circle*{5}} \put(20,-10){$p'$}
	\put(40,50){\circle*{5}} \put(10,25){$q'$}
		\put(40,0){\line(0,1){50}} \put(45,20){$d$}
	\put(150,0){\circle*{5}}
	\put(250,0){\circle*{5}}
	\put(210,100){\circle*{5}}
        \qbezier(150,0)(195,20)(250,0) \put(200,-5){$p$}
		\put(190,-5){\line(-5,-1){35}} \put(155,-12){\vector(-1,0){0}}
		\put(210,-5){\line(5,-1){35}} \put(245,-12){\vector(1,0){0}}
        \qbezier(150,0)(190,30)(210,100) \put(170,55){$q$}
		\put(167,45){\line(-4,-5){24}}\put(143,15){\vector(-1,-1){0}}
		\put(180,65){\line(4,5){20}} \put(200,90){\vector(1,1){0}}
        \qbezier(250,0)(220,30)(210,100) \put(225,60){$r$}
	\put(190,10){\circle*{5}} \put(170,-3){$p'$}
	\put(190,45){\circle*{5}} \put(160,20){$q'$}
		\put(190,10){\line(0,1){35}} \put(195,25){$\leq d$}
	\end{picture}
	\caption{Comparable triangles in Euclidean and CAT(0) spaces}
	\label{fig:cat0}
\end{figure}

The CAT(-1)\index{CAT(-1)} condition is defined similarly with respect
to the hyperbolic plane; any CAT(-1) space is $\delta$-hyperbolic, for some $\delta$, and hence CAT(-1) groups are word hyperbolic.

A (not necessarily geodesic) metric space $(X,d)$ is said to have non-positive curvature\index{non-positive curvature} (or curvature $\leq 0$) if every point of $X$ is contained in a CAT(0)
neighbourhood.
By the Cartan-Hadamard theorem \cite{BridsonHaefliger} the universal cover of a complete connected space of non-positive curvature is CAT(0).

Niblo and Reeves studied in particular groups acting on CAT(0) cube complexes\index{CAT(0)!cube complex}:
\begin{theorem}[Niblo, Reeves 1998 \cite{NibloReeves98}]
\label{thm:NRaction}
A group acting faithfully, properly discontinuously and cocompactly on
a simply connected and non-positively curved cube complex is biautomatic.
\end{theorem}
A cube complex\index{cube complex} is defined to be a metric polyhedral complex in which each cell is isometric
to the Euclidean cube with side lengths 1, where the gluing maps are isometries. 
Such a complex 
is non-positively curved provided that it contains 
at most one edge joining any two vertices, and no triangles of edges,
and (by a result of Gromov \cite{Gromov})  is CAT(0) if non-positively curved and simply connected.

Actions of Coxeter groups\index{Coxeter group}
on CAT(0) cube complexes are constructed in \cite{NibloReeves03}, but are not 
necessarily cocompact. However in some cases it follows from those or related 
constructions that the Coxeter groups are biautomatic (see Section~\ref{sec:coxeter}).

There are many open problems relating to CAT(0) groups\index{CAT(0)!group}
(see for example \cite{FarbHruskaThomas}).
The question of whether every CAT(0) group must be biautomatic
was recently resolved in the negative by Leary and Minasyan \cite{LearyMinasyan}, who constructed an example of a 3-dimensional CAT(0) group which could admit no biautomatic subgroup of finite index.  
It is still unknown whether non-automatic CAT(0) groups can exist.

However a restricted class of CAT(0) groups is provided by groups that act 
geometrically on CAT(0) spaces with isolated flats.
A $k$-flat in a CAT(0) space is an isometrically embedded copy of Euclidean space $\R^k$.
This family
contains a number of interesting examples, including geometrically finite Kleinian groups,  the fundamental groups of various compact manifolds, and limit groups,
arising from the solutions of equations over free groups. Groups of this type
are studied in \cite{HruskaKleiner}, where more details (of definition and examples) can be found.
Theorem 1.2.2 of that article establishes a number of properties of such 
groups, including their biautomaticity.

A form of non-positive curvature in simplicial complexes is defined in \cite{JS06}:
a flag simplicial complex $\cX$ is called $k$-systolic\index{systolic!complex} if
connected, simply connected and locally $k$-large (no minimal $\ell$-cycle
with $3<\ell<k$ in the link of a vertex).
A group is called $k$-systolic\index{systolic!group} if it acts simplicially, 
properly discontinuously and cocompactly on a $k$-systolic simplical complex, and is called systolic if 6-systolic.
\begin{theorem}[Januszkiewicz, Swiatkowski, 2006 \cite{JS06}]
\label{thm:JSaction}
7-systolic groups are hyperbolic, 6-systolic groups are biautomatic.
\end{theorem}
This result is used to prove biautomaticity of a large class of Artin groups\index{Artin group} \cite{HuangOsajda20}, as detailed in Section~\ref{sec:artin}.

A {\em Helly graph}\index{Helly!graph} is a graph in which every family of pairwise intersecting 
balls has a non-empty intersection. A group is called Helly\index{Helly!group} if it acts properly and cocompactly by graph automorphisms on a Helly graph; word hyperbolic groups\index{hyperbolic group}, CAT(0) cubical groups\index{CAT(0)!group} and C(4)-T(4) small cancellation\index{small cancellation} groups are all examples. It is proved in \cite{CCG+20} that all Helly groups are biautomatic. This result is used to 
 prove biautomaticity of another large class of Artin groups\index{Artin group} \cite{HuangOsajda21}, as detailed in Section~\ref{sec:artin}.

\section{Some automatic and biautomatic families}
\label{sec:eg_families}
Over a period of more than 30 years, automatic and biautomatic structures were found
for various families of groups, including braid groups, 
many Artin groups\index{Artin group}, mapping class groups\index{mapping class group}, and Coxeter groups\index{Coxeter group}. But some questions remain open for these families.

\subsection{Braid groups, Artin groups and Mapping Class groups}
\label{sec:artin}

Automatic structures for the braid group $B_n$ on $n$ strands and also for the (closely related) mapping class group\index{mapping class group} of the $(n+1)$-punctured sphere
were constructed by Thurston and are  described in \cite{ECHLPT}; 
one of the structures described for the braid groups is symmetric, proving 
the braid groups to be biautomatic. The automaticity (but not necessarily biautomaticity) of the mapping class group
of the $n+1$-punctured sphere then follows from the fact that it contains the quotient of the braid group $B_n$ by its centre as a subgroup of index $n+1$.

The braid group on $n+1$ strands is isomorphic to the Artin group of finite 
type $A_n$.
We recall that an Artin group\index{Artin group}
is a group defined by a presentation of the form
\[
\langle x_1,x_2,\cdots,x_n \mid \overbrace{x_ix_jx_i\cdots}^{\mij}= \overbrace{x_jx_ix_j \cdots}^{\mij},\, i\neq j \in \{1,2,\ldots,n\}\rangle,\]
relating to a symmetric, integer {\em Coxeter matrix} $(\mij)$, or equivalently 
a {\em Coxeter diagram}\index{Coxeter diagram} $\Gamma$ on $n$ vertices, whose edge $\{i,j\}$ is labelled
$\mij$, and is naturally associated with a
Coxeter group\index{Coxeter group} by adding relations $x_i^2=1$ for each $i$.
The Artin group has finite type if the associated Coxeter group is finite
(and hence $\Gamma$ is a disjoint union of diagrams from the well--known
list of spherical Coxeter diagrams).

In \cite{Charney92},
Charney used results of Deligne to extend 
Thurston's construction for the braid groups to all finite type Artin groups.
Charney's construction provided biautomatic structures for all
finite type Artin groups;
these biautomatic structures were geodesic over the ``Garside'' generating sets, 
but not over the standard generators $x_i$.
Biautomatic structures for all Garside groups
(of which finite type Artin groups are examples) were described by Dehornoy \cite{Dehornoy02}. 

For Artin groups of FC type (free products of finite type groups with
amalgamation over parabolic subgroups, for which the complete subgraphs of the
labelled graph formed by deleting all $\infty$-labelled edges from $\Gamma$ are all of finite type), {\em asynchronously automatic} structures
were constructed in \cite{Altobelli}, and used to define quadratic time solutions to the word problem\index{word problem}; we recall that an exponential (rather then quadratic) time
solution is guaranteed by asynchronous automaticity.
Right-angled Artin groups (those for which all the parameters $\mij$ are
within the set $\{2,\infty\}$, which form a subset of FC type) were then proved automatic in
\cite{HermillerMeier,VanWyk}. Very recently \cite{HuangOsajda21} Artin groups 
of FC type have been proved to be Helly\index{Helly!group}, and hence biautomatic.

Mosher's paper \cite{Mosher95} answered a major open question raised by
Thurston's proof of the automaticity of the mapping class group\index{mapping class group} of the punctured
sphere. Using quite different techniques from Thurston, Mosher proved automaticity of the mapping class group of any surface of finite type, that is, the group of (orientation preserving) homeomorphisms modulo isotopy of any surface obtained from a compact surface by removing at most finitely many points.
In the case of a surface with at least one puncture the automatic structure is explicitly 
defined (and could be constructed),
in terms of a complex
whose vertices are {\em ideal triangulations} on S (triangulations with vertex set the puncture set) and whose edges are elementary moves between ideal triangulations. 
The more general case can be reduced to the case of a punctured surface using a short exact sequence.
The question of whether
the mapping class group\index{mapping class group} was in fact biautomatic was finally solved by Hamenstaedt's construction of a biautomatic structure in 2009 \cite{Hamenstaedt}.

An Artin group is defined to have {\em large type} if all the associated 
parameters $\mij$ are at least 3, {\em extra large type}
if all $\mij$ are at least 4. For large and especially extra large groups 
small cancellation techniques associated with negatively curved geometry
were developed in \cite{AppelSchupp}. 
All extra large Artin groups
were proved biautomatic in \cite{Peifer},
using those small cancellation techniques; the language is a set of geodesics over the standard generating set.
All those groups and many others of large type were
found by 
Brady and McCammond \cite{BradyMcCammond} to act appropriately on piecewise 
Euclidean non-positively curved 2-complexes of types $A_2$ or $B_2$,
and hence, by results of
\cite{GerstenShort90,GerstenShort91} to be biautomatic (but in this case
the biautomatic structure is defined over a non-standard generating set).

All Artin groups of large type
were proved to be shortlex automatic
over their standard generating sets in \cite{HoltReesSLArtin}.
A rewrite system was described, which rewrote any word to shortlex geodesic
form using sequences of moves on 2-generator substrings. The result extended 
beyond large type to {\em sufficiently large} type, where some 
parameters $\mij$ might take the value $2$ (provided that for any triple $i,j,k$,
if $\mij=2$, then either $\mik=\mjk=2$ or at least one of $\mik$ and $\mjk$ is infinite). Biautomatic structures for
all large type Artin groups (and in fact for the slightly large class of {\em almost large} groups)
were proved to exist in \cite{HuangOsajda20}, where all those
groups were proved to have appropriate actions on {\em systolic}
complexes\index{systolic!complex}.
An Artin group is called almost large if for any triple $i,j,k$ it is only possible to have $\mij=2$ if one of $\mik$ or $\mjk$ is infinite, and for any 4-set $i,j,k,l$
at most 2 of $\mij$, $\mjk$, $\mkl$, $\mil$ can be equal to 2 unless one of the four parameters is infinite.

\subsection{Coxeter groups}
\label{sec:coxeter}

The proof in \cite{BrinkHowlett} of shortlex automaticity of any Coxeter group\index{Coxeter group} relative to its standard generating set provided a result that had long been conjectured. We recall that a Coxeter group $W$ is
described by a presentation
\[  \langle x_1,\ldots,x_n \mid x_i^2 = 1,\,(x_ix_j)^{\mij} = 1,\,i\neq j \in \{1,\ldots,n\} \rangle, \]
relative to a Coxeter matrix $(\mij)$ and associated Coxeter diagram\index{Coxeter diagram} $\Gamma$;
the set 
$X=\{x_1,\ldots, x_n\}$ is its {\em standard generating set}.

The proof of the theorem constructs an automatic structure for $W$ using properties of its
associated {\em root system}, which arises from the natural isomorphism 
between $W$ and a reflection group $\overline{W}$ as we now describe;
more details can be found in \cite{Humphreys}.
The group $\overline{W}$ is generated by  a set of {\em reflections} $r_1,\ldots,r_n$ of $\R^n$ defined by
$r_i(v) := v - 2 \langle v, e_i \rangle e_i$, for $v \in \R^n$,
where $e_i: i=1,\ldots,n$ is a basis for $\R^n$ and $\langle,\rangle$ is the 
symmetric, bilinear form on $\R^n$ defined by $\langle e_i,e_j\rangle = - \cos(\pi/\mij)$. 
The isomorphism from $W$ to $\overline{W}$ maps $x_i$ to $r_i$, and induces an action of $W$ on $\R^n$. 
The roots of $W$ are defined to be the elements of the set $\Phi=W\{e_1,\ldots,e_n\}$,
which decomposes as a disjoint union $\Phi^+ \cup \Phi^-$ of {\em positive roots} (vectors $\sum \lambda_ie_i$ with all $\lambda_i \geq 0$) and their negatives.

Brink and Howlett's proof of regularity of the set of shortlex geodesic words in $W$ is derived from their proof in \cite{BrinkHowlett} of the 
finiteness of the set of positive roots for $W$ that {\em dominate} any given positive root;
a positive root $\alpha$ is said to dominate a second positive root $\beta$
if whenever $w(\alpha)$ is negative, for $w \in W$, then so is $w(\beta)$.
We define $\widetilde{\Delta}_W$ to be the set
of positive roots that dominate no others.
Then a word acceptor\index{word acceptor} $\WA$ for a shortlex automatic structure for $W$ 
can be built whose accepting states are all subsets of 
$\widetilde{\Delta}_W$ \cite[Proposition 3.3]{BrinkHowlett}.

The transitions in  $\WA$ are determined by the following observation from
\cite[Lemma 3.1]{BrinkHowlett}. When 
$w=x_{i_1}\cdots x_{i_l}$ is a shortlex geodesic word representing an element of $W$
then, for $x_i \in X$, the word $w'=wx_i$ is non-geodesic if and only if there exists
$j \in \{1,\ldots,l\}$ for which $e_i = x_{i_l}\cdots x_{i_{j+1}}(e_{i_j})$.  
In the case where $w'=wx_i$ is geodesic, that fails to be shortlex minimal 
if and only if there exists 
$j \in \{1,\ldots,l\}$ and a generator $x_k \prec x_{i_j}$ for which $e_i = x_{i_l}\cdots x_{i_j}(e_k)$.  
In that case the word $x_{i_1}\cdots x_{i_{j-1}}x_kx_{i_j}\cdots x_{i_l}$ is shortlex minimal.
Based on these two facts, 
transition on a generator $x_i$ from (the state corresponding to) a subset $S$ of $\widetilde{\Delta}_W$ is to a failure state $F$ if $e_i \in S$. 
But for $e_i \not\in S$, transition is to the intersection with 
$\widetilde{\Delta}_W$ of the set
\[ S'' = \{ x_i(\alpha) \mid \alpha \in S\} \cup \{e_i\} \cup \{ x_i(e_k) \mid x_k \prec x_i \}. \] 

A similar construction to the above, described in \cite{Howlett}, proves regularity of the set of all geodesic words in $G$
over $S$.

The question of whether all Coxeter groups are not just automatic but actually biautomatic remains open.
Work of Niblo and Reeves \cite{NibloReeves03} shows that any finitely generated
Coxeter group $G$ acts properly discontinuously by isometries on a locally finite, 
finite dimensional CAT(0) cube complex\index{CAT(0)!cube complex};  their construction is based on the root system $\Phi$ associated with $G$, and an extension of the 
dominance relation of \cite{BrinkHowlett} from $\Phi^+$ to $\Phi$.
When the action of $G$ on the cube complex is cocompact, 
then biautomaticity follows, using \cite{NibloReeves98}.
Cocompact actions are proved in \cite{NibloReeves03} to exist 
whenever
$G$ is right-angled or word hyperbolic
(by \cite{Moussong} word hyperbolicity of $G$ is recognisable from the diagram $\Gamma$). 
It is also observed in \cite{NibloReeves03} that,
by \cite{Wil99}, cocompact actions are guaranteed whenever $G$ contains only finitely many conjugacy classes of
subgroups isomorphic to rank 3 parabolic subgroups  
$\langle x_i,x_j,x_k \rangle$ 
(associated with rank 3 subdiagrams $\Gamma_{ijk}$ of $\Gamma$)
for which $\mij,\mik,\mjk$ are all finite; 
\cite{CapraceMuhlherr} used this result to derive biautomaticity of $G$ provided that $\Gamma$ contains no 
affine subdiagram of rank 3 or more. Subsequently, Caprace \cite{Caprace} 
proved biautomaticity of all relatively hyperbolic\index{hyperbolic group!relatively} Coxeter groups using results from \cite{HruskaKleiner}.

The dimension of a Coxeter group is defined to be the dimension of its {\em Davis complex},
equivalently the maximal rank of any of its spherical parabolic subgroups.
It follows that
a Coxeter group is 2-dimensional\index{Coxeter group!2-dimensional}
if none of the rank 3 subdiagrams $\Gamma_{ijk}$ is spherical, equivalently if for all $i,j,k$, 
$\frac{1}{\mij}+\frac{1}{\mik}+\frac{1}{\mjk} \leq 1$. The biautomaticity 
of all 2-dimensional Coxeter groups is proved in \cite{MunroEtAl}. The construction of a geodesic language generalises ideas from \cite{NibloReeves03},
and the result generalises an earlier result proving biautomaticity for certain 2-dimensional groups that used the results of \cite{NibloReeves03}.
 
\section{Open problems}
\label{sec:open_problems}

More than 30 years after the subject started there continue to be many open problems involving automatic groups. 
Some of these problems date from the beginning of the subject, and are listed 
in \cite{ECHLPT}. Some but not all of these have been mentioned within this chapter.
In particular, it remains open whether automatic groups exist that are not biautomatic (see Section~\ref{sec:automatic})
also whether automatic groups exist that do not have soluble conjugacy problem\index{conjugacy problem} (see Section~\ref{sec:properties})
,
whether all soluble automatic groups must be virtually abelian.
The most recent progress on this last question was made by 
the proof of Romankov \cite{Romankov}, 
that a soluble biautomatic group must be virtually abelian (see Section~\ref{sec:soluble}).

It is still unknown whether a non-biautomatic Coxeter group can exist (Section~\ref{sec:coxeter}),
or a non-automatic Artin group (Section~\ref{sec:artin}).

There are many open problems relating to group actions, in particular, whether  a CAT(0) group\index{CAT(0)!group} must be
automatic. The very recent construction in \cite{LearyMinasyan} of a 3-dimensional CAT(0) group that cannot be biautomatic (Section~\ref{sec:actions}) represents a major advance on this problem; it does not resolve the question of
automaticity. The question of whether biautomaticity or automaticity are implied
for a 2-dimensional, piecewise Euclidean CAT(0) group remains open (but we note the recent 
contribution to this problem of the main result of \cite{MunroEtAl}, see 
Section~\ref{sec:coxeter}). The 2-dimensional problem is number 43 on a list of open 
problems within geometric group theory that was published ten years ago 
in \cite{FarbHruskaThomas}, and motivated a body of research, and rapid 
solution of some of the problems. However, some of 
the problems listed in this useful and extensive list, or in the earlier list \cite{Bestvina}, remain open.

\end{document}